\documentclass{article} 
 \usepackage{amsmath}
\usepackage{tikz-cd}
\usepackage{tikz}
\usetikzlibrary{cd,decorations.pathmorphing}
\usepackage{listings}
\usepackage[margin = 1in]{geometry}

\usepackage{xypic}
\usepackage{amsthm}
 \usepackage{amssymb} 
 \usepackage{hyperref}
 \usepackage{textcomp} 
\usepackage{enumerate}
\usepackage{caption}

\usepackage{latexsym}
\usepackage[OT2,T1]{fontenc}
\newtheorem{remark}{Remark}

\usepackage{verbatim}
\usepackage{graphicx}
\usepackage{nomencl}
\DeclareFontEncoding{OT2}{}{} 

\newcommand{\Z}{\mathbf{Z}}

\newcommand{\Aut}{\mathrm{Aut}}

 \newcommand{\HH}{\mathbf{H}}

\DeclareMathAlphabet{\mathscr}{OT1}{rsfs}{n}{it}

\newtheorem{theorem}{Theorem}

\newtheorem{lemma}{Lemma}

\newcommand{\ZZ}{\boldsymbol{\zeta}}
\newcommand{\KL}{\mathrm{KL}}

\newtheorem{heuristic definition}{Heuristic definition}

\title{Entropy of Cohen-Lenstra measures: the $u$-aspect}
\author{Artane Siad}
\date{March 17, 2024}

\begin{document}  

\maketitle

 \begin{abstract}
Let ${\rm \mathbf{H}}(\nu^{u}_{\rm CL})$ be the entropy of the Cohen-Lenstra measure on finite abelian $p$-groups associated to an integral unit-rank $0 \le u \in \mathbb{N}$. In this note, we show that $0 < {\rm \mathbf{H}}(\nu^{u}_{\rm CL}) < \infty$ for all $u$, ${\rm \mathbf{H}}(\nu^{u}_{\rm CL})$ is a strictly decreasing function of $u \ge 0$, and ${\rm \mathbf{H}}(\nu^{u}_{\rm CL}) \xrightarrow{u \to \infty} 0$.
In particular, this shows that the groupoid measure is an entropy maximizer in the class of Cohen-Lenstra measures of varying integral unit-rank on finite abelian $p$-groups. 
 \end{abstract}

\section{Entropy and statement}

Let $(X,\nu)$ be a discrete probability space. The \textbf{Shannon entropy} of $\nu$ is defined as the expected value of $-\log \nu$
$${\rm \mathbf{H}}(\nu) := \mathbb{E}(\log \nu) = - \sum_{x \in X} \nu(x) \log \nu(x) \ge 0$$
and is a measure of the information content of the measure $\nu$ \cite{MR4679166}.

Fix a prime $p$ and let ${\rm FinAb}_{p}$ denote the category of finite abelian $p$-groups. 
The Cohen-Lenstra measure on ${\rm FinAb}_{p}$ associated to an integral unit-rank $0 \le u \in \mathbb{N}$,  $\nu^{u}_{\rm CL}$,  is the groupoid measure on ${\rm FinAb}_{p}$ quotiented by $u$ randomly-chosen elements. Alternatively, $\nu^{u}_{\rm CL}$ is characterized by the property 
$$\nu^{u}_{\rm CL}(A) \propto \frac{1}{\#A^{u} \#\Aut A}$$
for all finite abelian $p$-groups $A \in {\rm FinAb}_{p}$. 

The purpose of this note is to establish some basic properties of ${\rm \mathbf{H}}(\nu^{u}_{\rm CL})$: we show that ${\rm \mathbf{H}}(\nu^{u}_{\rm CL})$ is finite for all $u$, strictly decreasing as a function of $u$, and converges to $0$ as $u$ approaches $\infty$.

\begin{theorem} \label{thm:main-entropy}
Let ${\rm \mathbf{H}}(\nu^{u}_{\rm CL})$ denote the entropy of the Cohen-Lenstra measure $\nu^{u}_{\rm CL}$ on finite abelian $p$-groups associated to unit-rank $0 \le u \in \mathbb{N}$. Then 
\begin{itemize}
\item[\rm{I)}] ${\rm \mathbf{H}}(\nu^{u}_{\rm CL}) < \infty$ for all $u$;
\item[\rm{II)}]  ${\rm \mathbf{H}}(\nu^{u}_{\rm CL})$ is a strictly decreasing function of $u \ge 0$; and, 
\item[\rm{III)}]  ${\rm \mathbf{H}}(\nu^{u}_{\rm CL}) \xrightarrow{u \to \infty} 0$.
\end{itemize}
\end{theorem}

We view our statement as a first step in introducing the Principle of Maximum Entropy in the study of Cohen-Lenstra measures. Such a principle has been profitably exploited in probability, see for instance, the information theoretic proofs of the Central Limit Theorem \cite{MR0124081,MR0815975}.  
In particular, Theorem \ref{thm:main-entropy} shows that the groupoid measure is an entropy maximizer for Cohen-Lenstra measures of varying integral unit-rank on finite abelian $p$-groups. 

\begin{remark}
There are variants of the Cohen-Lenstra measures where the integrality assumption is dropped and the parameter $u$ is taken to be any real number $> -1$. The estimates contained in our proof are sufficient to deduce items I) and III) of Theorem \ref{thm:main-entropy} for these variants. 
\end{remark}

We also obtain an explicit formula for the relative entropy between Cohen-Lenstra measures. 
The precise definition of the relative entropy, also called the Kullback-Leibler Divergence, $D_{\KL} \left(\nu^{u_{1}}_{\rm CL} \,\, \lvert \rvert \,\, \nu^{u_{2}}_{\rm CL} \right)$, will be given in \S \ref{sec:relative-entropy}.

\begin{theorem} \label{thm:main-relative}
Let $\nu^{u_{1}}_{\rm CL}$ and $\nu^{u_{2}}_{\rm CL}$ be Cohen-Lenstra measures on finite abelian $p$-groups associated to unit-ranks $u_{1} \ge 0 $ and $u_{2} \ge 0$ respectively. The relative entropy of $\nu^{u_{1}}_{\rm CL}$ from $\nu^{u_{2}}_{\rm CL}$ is given by:
$$D_{\KL} \left(\nu^{u_{1}}_{\rm CL} \,\, \lvert \rvert \,\, \nu^{u_{2}}_{\rm CL} \right) = 
\log \left( \frac{F_{u_{1}}}{F_{u_{2}}} \right)  
+ (u_{2}-u_{1}) \sum_{i \ge 1} \frac{\log(p)}{p^{u_{1}+i}-1} $$
where $F_{u}$ is the normalizing constant $\prod_{i \ge 1+u} (1-p^{-i})$ for the Cohen-Lenstra measure $\nu^{u}_{\rm CL}$.
\end{theorem}

The proof exploits explicit formulas for the Cohen-Lenstra zeta functions $\ZZ_{k}^{(p)}(s)$. 

\begin{remark}
An approach to finding explicit formulas for the Shannon entropy of Cohen-Lenstra measures would be to gain a good understanding of a variant of the Cohen-Lenstra zeta function $ \ZZ_{k}^{(p)}(s)$ with a power of $s$ not on the $1/\# A$ terms, but on the $1/\#Aut$ terms instead. An explicit expression for the entropy would then fall out by taking the derivative of this new zeta function. 
We hope to supplement a future version of the present paper with such an explicit formula for the Shannon entropy. 
\end{remark}

\section{Preliminaries}

\emph{Phillips Hall's Strange Formula} for finite abelian $p$-groups states that
$$\sideset{}{'}\sum_{A} \frac{1}{\# \Aut A} = \sideset{}{'} \sum_{A}\frac{1}{\#A} = \sum_{n} \frac{\pi(n)}{p^{n}} = \prod_{i \ge 1} (1-p^{-i})^{-1} < \infty$$ 
where the sums $\sum'$ run over isomorphism classes of finite abelian $p$-groups and $\pi$ is the partition function \cite{MR1509594,MR1181111}.
This formula, and its variants, imply the following description of the Cohen-Lenstra measures
$$\nu^{u}_{\rm CL}(A) = \frac{1}{\#A^{u} \#\Aut A}\prod_{i \ge 1} (1-p^{-u-i}).$$
Let us denote by $F_{u}$ the normalizing constant $\prod_{i \ge 1} (1-p^{-u-i}) = \prod_{j \ge u+1} (1-p^{-j})$ (see \cite{MR0756082}) and note that $F_{u}$ is a strictly increasing function of $u$. \\

The number of automorphisms of a finite abelian group is comparable, and often much larger, than the size of the group. The following lemma will be useful and is a lower bound expressing this fact. 
\begin{lemma} \label{lem:aut-lower-bound}

For a finite abelian $p$-group $A$ of size $p^{n}$ we have $$\# \Aut(A) \ge  \# A (1-p^{-1}) \ge p^{n-1}.$$
Furthermore, $\#\Aut A \ge \# A = p^{n}$ whenever $\mathrm{rank}_{p}(A) \ge 2$. 

\end{lemma} 
\begin{proof}
Denote by $A_{\lambda'} = \prod_{i} \Z_{p}/p^{\lambda_{i}'}\Z_{p}$ the finite abelian $p$-group of associated to the partition $\lambda' = (\lambda_{i}')_{i} = (\lambda_{1}' \ge \lambda_{2}' \ge \ldots )$. Recall that the number of automorphisms of a finite abelian $p$-group of type $\lambda'$ is
$$\# \Aut A = p^{|\lambda'|+2n(\lambda')} \prod_{j \ge 1} \prod_{k=1}^{\lambda_{j}-\lambda_{j+1}}(1-p^{-k})$$
where $\lambda$ denotes the dual partition of $\lambda'$, $|\lambda'| = \sum_{j \ge 1} \lambda_{j}' = n$, and $n(\lambda') = \sum_{i} (i-1)\lambda_{i}' = \sum {\lambda_{j} \choose 2}$ \cite{MR1354144}.
To find the lower bound for $\# \Aut(A)$ when $\# A = p^{n}$, we rewrite this expression as follows

\begin{equation} \label{eqn:aut-expression}
\# \Aut A = p^{n} p^{\sum_{j \ge 1} (\lambda_{j}^{2}-\lambda_{j})-\frac{(\lambda_{j}-\lambda_{j+1})^{2}}{2} - \frac{(\lambda_{j}-\lambda_{j+1})}{2}} \prod_{j \ge 1} \prod_{k=1}^{\lambda_{j}-\lambda_{j+1}}(p^{k}-1),
\end{equation}
where the $(\lambda_{j}^{2}-\lambda_{j})$ terms come from $n(\lambda')$ and the terms $-\frac{(\lambda_{j}-\lambda_{j+1})^{2}}{2} - \frac{(\lambda_{j}-\lambda_{j+1})}{2}$ come from normalizing the terms $ \prod_{k=1}^{\lambda_{j}-\lambda_{j+1}}(p^{k}-1)$ to have positive powers of $p$. 

Let's analyze the exponent. We use the notation $\lambda^{2}$ to denote the partition $(\lambda_{1}^{2} \ge \lambda_{2}^{2} \ge \ldots)$ where each component of $\lambda$ is squared. We write $m_{j} := \lambda_{j} - \lambda_{j+1}$. We have:
\begin{align*}
&\sum_{j \ge 1} (\lambda_{j}^{2}-\lambda_{j})-\frac{(\lambda_{j}-\lambda_{j+1})^{2}}{2} - \frac{(\lambda_{j}-\lambda_{j+1})}{2} \\
&= |\lambda^{2}| - |\lambda| - \frac{|\lambda^{2}|}{2} - \left( \frac{|\lambda^{2}|}{2} - \frac{\lambda_{1}^{2}}{2} \right) - \frac{|\lambda|}{2} + \left( \frac{|\lambda|}{2} - \frac{\lambda_{1}}{2} \right) + \sum_{j \ge 1} \lambda_{j} \lambda_{j+1} \\
&= \frac{\lambda_{1}^{2}}{2} - \frac{\lambda_{1}}{2} - |\lambda| +  \sum_{j \ge 1} \lambda_{j} \lambda_{j+1} \\
&= \frac{\lambda_{1}^{2}}{2} - \frac{\lambda_{1}}{2} - |\lambda| +  \sum_{j \ge 1} (\lambda_{j+1}+m_{j}) \lambda_{j+1} \\
&= \frac{\lambda_{1}^{2}}{2} - \frac{\lambda_{1}}{2} - |\lambda| + |\lambda^{2}|-\lambda_{1}^{2} + \sum_{j \ge 1} m_{j} \lambda_{j+1} \\
&=\frac{\lambda_{1}^{2}}{2} - \frac{3\lambda_{1}}{2} + \left( (|\lambda^{2}|-\lambda_{1}^{2}) - (|\lambda|-\lambda_{1}) \right) + \sum_{j \ge 1} m_{j} \lambda_{j+1}. 
\end{align*}
The only term which can be negative is $\lambda_{1}^{2}/2 - 3\lambda_{1}/2$. It is in fact non-negative, except for $\lambda_{1} = 1,2$ in which case it is equal to $-1$. 

Next, to prove that $\#\Aut A \ge \# A$ whenever $\mathrm{rank}_{p}(A) \ge 2$, note that when $\lambda_{1} = 2$, then either one of the $\lambda_{j} - \lambda_{j+1} = 2$ or two of them are $= 1$. In both cases, $p^{-1}$ times the corresponding terms $\prod_{k=1}^{\lambda_{j}-\lambda_{j+1}}(p^{k}-1)$ in Equation (\ref{eqn:aut-expression}) is $\ge 1$, except when $p=2$ and two of the $\lambda_{j} - \lambda_{j+1} =1 $. In that last case, the term $\sum_{j \ge 1} m_{j} \lambda_{j+1}$ is $\ge 1$.

If $\lambda_{1} = 1$, then $A$ is a cyclic $p$-group, whence $\# \Aut A = (p^{n}-p^{n-1}) = p^{n} (1-p^{-1}) = \# A (1-p^{-1})$. This completes the proof. 
\end{proof}

\section{Proof of Theorem \ref{thm:main-entropy}}

With the preliminaries in hand, we turn to the proof.

\begin{enumerate}[I)]
\item To show that ${\rm \mathbf{H}}(\nu^{u}_{\rm CL})$ is finite for all $u \ge 0$, it will suffice to show that ${\rm \mathbf{H}}(\nu^{0}_{\rm CL})$ is finite as we will show in II) that ${\rm \mathbf{H}}(\nu^{u}_{\rm CL})$ is strictly decreasing in $u$.
Now, using $\log(x) \ll_{\varepsilon} x^{\varepsilon}$, we have the following estimates
\begin{align*}
{\rm \mathbf{H}}(\nu^{0}_{\rm CL}) &= - \sideset{}{'}\sum \frac{F_{0}}{\# \Aut A} \log \left( \frac{F_{0}}{\# \Aut A} \right)   \\
&\ll_{\varepsilon} -\log(F_{0}) + F_{0}\sideset{}{'} \sum_{A \ne 1}\frac{1}{(\# \Aut A)^{1-\varepsilon}} \\
&= -\log(F_{0}) + F_{0}\sum_{n \ge 1}\sideset{}{'} \sum_{\#A = p^{n}} \frac{1}{(\# \Aut A)^{1-\varepsilon}} \\
&\le -\log(F_{0}) + F_{0}\sum_{n \ge 1} \frac{\pi(n)}{(p^{n- 1 })^{1-\varepsilon}} < \infty
\end{align*}
where $\pi(\cdot)$ is the partition function. The sum on the last line is finite by the root test since 
$$\frac{\pi(n)}{(p^{n- 1 })^{1-\varepsilon}} \sim \frac{\frac{1}{4n\sqrt{3}}e^{\pi \sqrt{\frac{2n}{3}}}}{\left(p^{n- 1}\right)^{1-\varepsilon}}$$
as $n \rightarrow \infty$, implying that $\left( \frac{\pi(n)}{(p^{n- 1 })^{1-\varepsilon}} \right)^{\frac{1}{n}} \xrightarrow{n \to \infty} \frac{1}{p^{1-\epsilon}} < 1$. 

\begin{remark}
Running this argument with the identity ${\rm \HH}(\nu^{u}_{\rm CL}) = - \log(F_{u}) + F_{u} \sideset{}{'}\sum_{A \ne 1} \frac{\log( \#A^{u} \# \Aut A)}{\#A^{u} \#\Aut A}$ shown below gives that ${\rm \HH}(\nu^{u}_{\rm CL}) < \infty$ for all real $u > -1$. 
\end{remark}

\item We have:
\begin{align*}
{\rm \HH}(\nu^{u}_{\rm CL}) &= - \sideset{}{'}\sum \nu^{u}_{\rm CL} \\ 
&= - \sideset{}{'}\sum \frac{F_{u}}{\#A^{u} \# \Aut A} \log \left( \frac{F_{u}}{\#A^{u} \# \Aut A} \right)  \\
&=  - \sideset{}{'}\sum \frac{F_{u}}{\#A^{u} \# \Aut A} \log(F_{u}) + \sideset{}{'}\sum \frac{F_{u}}{\#A^{u} \# \Aut A} \log (\#A^{u} \# \Aut A) \\
&= - \log(F_{u}) + F_{u} \sideset{}{'}\sum_{A \ne 1} \frac{\log( \#A^{u} \# \Aut A)}{\#A^{u} \#\Aut A} 
\end{align*}
Since $F_{u}$ is a strictly increasing function of $u$, the term $- \log(F_{u})$ is strictly decreasing. 
It would then be sufficient to show that the terms $F_{u} \frac{\log( \#A^{u} \# \Aut A)}{\#A^{u} \#\Aut A}$ are individually decreasing 
(note the tension since $F_{u}$ is strictly increasing). This reduces to proving:
\begin{equation} \label{eqn:decreasing-terms}
\#A^{u+1} \# \Aut A \le (\#A^{u} \# \Aut A)^{(1-p^{-(u+1)})\#A}
\end{equation}
for $A \ne 1$ and $u \ge 0$.

When $p \ge 3$ and $u \ge 1$ we have: 
$$ \#A^{u+1} \# \Aut A \le (\#A^{u} \# \Aut A)^{p-1} \le  (\#A^{u} \# \Aut A)^{(1-p^{-(u+1)})\#A}.$$

When $p = 2$ and $u \ge 1$, we have: 
$$ \#A^{u+1} \# \Aut A \le (\#A^{u} \# \Aut A)^{\frac{3}{4}\#A} \le  (\#A^{u} \# \Aut A)^{(1-p^{-(u+1)})\#A}.$$
except when $A = \Z/2$ and $u = 1$. 

Now, when $u = 0$, we need to show that:
$$ \#A \# \Aut A \le  (\# \Aut A)^{(1-p^{-1})\#A}.$$

The reduces to 
$$ \#A \le  (\# \Aut A)^{\#A-1-\#A/p}$$

Using Lemma \ref{lem:aut-lower-bound}, we see that this is true except:
if $A = \Z/2$, for which $(\# \Aut A)^{\#A-1-\#A/p} = 1$,
if $A = \Z/4$, for which $(\# \Aut A)^{\#A-1-\#A/p} = 2$, or
if $A = \Z/3$, for which $(\# \Aut A)^{\#A-1-\#A/p} = 2$.

Thus, we have that each term is individually decreasing with the following four exceptions: 
\begin{enumerate}
\item when $A = \Z/2$ and $u = 0$;
\item when $A = \Z/4$ and $u = 0$;
\item when $A = \Z/2$ and $u = 1$;
\item when $A = \Z/3$ and $u = 0$. 
\end{enumerate} 

For these, one checks directly that: 
{\small
\begin{align*}
&\log(F_{0}) + F_{0} \frac{\log(\# \Aut \Z/2)}{\#\Aut \Z/2} + F_{0} \frac{\log(\# \Aut \Z/4)}{\#\Aut \Z/4} - \log(F_{1}) - F_{1}\frac{\log( \#\Z/2 \# \Aut \Z/2)}{\# \Z/2 \#\Aut \Z/2} - F_{1} \frac{\log(\Z/4 \# \Aut \Z/4)}{\Z/4 \#\Aut \Z/4} \ge 0.44 \\
&\log(F_{1}) + F_{1} \frac{\log(\#\Z/2 \#\Aut \Z/2)}{\#\Z/2 \#\Aut \Z/2} - \log(F_{2}) - F_{2} \frac{\log((\#\Z/2)^{2} \#\Aut \Z/2)}{(\#\Z/2)^{2} \#\Aut \Z/2} \ge 0.21 \\
&\log(F_{0}) + F_{0} \frac{\log(\#\Aut \Z/3)}{\#\Aut \Z/3} - \log(F_{1}) - F_{1} \frac{\log(\#\Z/3 \#\Aut \Z/3)}{\#\Z/3 \#\Aut \Z/3} \ge 0.34
\end{align*}
}

by using the bound $(1-p^{-k})^{p/(p-1)} \le \prod_{j=k}^{\infty} (1-p^{-j}) \le 1$ (obtained using the concavity of $\log(1-t)$) to estimate the $F_{u}$ terms. We conclude that ${\rm \HH}(\nu^{u}_{\rm CL})$ is a strictly decreasing function of $u \in \mathbb{N}$.

\item Lastly, by the expression for ${\rm \HH}(\nu^{u}_{\rm CL})$ obtained in the proof of item II) above, and the fact that $\log(x) \le x$, we have the bound 
$${\rm \HH}(\nu^{u}_{\rm CL}) = - \log(F_{u}) + F_{u} \sideset{}{'}\sum_{A \ne 1} \frac{\log( \#A^{u} \# \Aut A)}{\#A^{u} \#\Aut A}  \le - \log(F_{u}) + F_{u} \sideset{}{'}\sum_{A \ne 1} \frac{u}{\#A^{u-1}\#\Aut A} + F_{u} \sideset{}{'}\sum_{A \ne 1} \frac{1}{\#A^{u}}$$
for $u \ge 2$. Now, since $\# A \ge p$ for $A \ne 1$, we obtain 
\begin{align*}
{\rm \HH}(\nu^{u}_{\rm CL}) &\le - \log(F_{u}) + \frac{uF_{u}}{p^{u-1}} \sideset{}{'}\sum_{A \ne 1} \frac{1}{\#\Aut A} + \frac{F_{u}}{p^{u-1}} \sideset{}{'}\sum_{A \ne 1} \frac{1}{\#A}  \\
&= \sum_{k \ge 1} \frac{1}{k} \frac{1}{(p^{k}-1)p^{ku}} + \frac{uF_{u}}{p^{u-1}} \sideset{}{'}\sum_{A \ne 1} \frac{1}{\#\Aut A} + \frac{F_{u}}{p^{u-1}} \sideset{}{'}\sum_{A \ne 1} \frac{1}{\#A}  \xrightarrow{u \to \infty} 0.
\end{align*}
Because the Shannon entropy is non-negative, we get ${\rm \HH}(\nu^{u}_{\rm CL})  \xrightarrow{u \to \infty} 0$. 
\end{enumerate}

\section{Proof of Theorem \ref{thm:main-relative} and relative entropy} \label{sec:relative-entropy}

Let $\nu$ and $\mu$ be two discrete probability measures on $X$ with the property that $\mu$ is absolutely continuous with respect to $\nu$,  $\mu \ll \nu$. The \textbf{relative entropy}, also called the Kullback--Leibler divergence, of $\nu$ from $\mu$ is defined as the expected value of $\log(\mu/\nu)$
$$D_{\KL}(\mu \,\, \lvert \rvert \,\, \nu) := \mathbb{E}_{\mu}\big(\log (\mu/\nu) \big) = \sum_{x \in X} \mu(x) \log \left( \frac{\mu(x)}{\nu(x)} \right) \ge 0$$
where we interpret contributions of terms with $\mu(x) = 0$ as $0$.
The relative entropy is non-negative by Gibbs' inequality.
The relative entropy measures the informational content of $\nu$ from the point of view of $\mu$.

\begin{theorem}
Let $\nu^{u_{1}}_{\rm CL}$ and $\nu^{u_{2}}_{\rm CL}$ be Cohen-Lenstra measures on finite abelian $p$-groups associated to unit-ranks $u_{1} \ge 0 $ and $u_{2} \ge 0$ respectively. The relative entropy of $\nu^{u_{1}}_{\rm CL}$ from $\nu^{u_{2}}_{\rm CL}$ is given by:
$$D_{\KL} \left(\nu^{u_{1}}_{\rm CL} \,\, \lvert \rvert \,\, \nu^{u_{2}}_{\rm CL} \right) = 
\log \left( \frac{F_{u_{1}}}{F_{u_{2}}} \right)  
+ (u_{2}-u_{1}) \sum_{i\ge1} \frac{\log(p)}{p^{u_{1}+i}-1} $$
where $F_{u}$ denotes the normalizing constant $\prod_{i \ge 1+u} (1-p^{-i})$.
\end{theorem}

\begin{proof}
The definition of $D_{\KL} \left(\nu^{u_{1}}_{\rm CL} \,\, \lvert \rvert \,\, \nu^{u_{2}}_{\rm CL} \right)$ gives
\begin{align}
D_{\KL} \left(\nu^{u_{1}}_{\rm CL} \,\, \lvert \rvert \,\, \nu^{u_{2}}_{\rm CL} \right) &= \sideset{}{'}\sum_{A} \nu^{u_{1}}_{\rm CL}(A) \log \left(\frac{\nu^{u_{1}}_{\rm CL} (A)}{\nu^{u_{2}}_{\rm CL}(A)} \right)  \\
&= \sideset{}{'}\sum_{A} \frac{F_{u_{1}}}{\# A^{u_{1}} \# \Aut A} \log \left( \frac{F_{u_{1}} \#A^{u_{2}}}{F_{u_{2}} \#A^{u_{1}}} \right)\\
&= \log \left( \frac{F_{u_{1}}}{F_{u_{2}}} \right) + F_{u_{1}} (u_{2}-u_{1}) \sideset{}{'}\sum_{A} \frac{\log(\#A)}{\# A^{u_{1}} \# \Aut A}   \\
&=  \log \left( \frac{F_{u_{1}}}{F_{u_{2}}} \right) + (u_{2}-u_{1})F_{u_{1}} \left( - \lim_{k \rightarrow \infty} \left. \frac{\mathrm{d}}{\mathrm{d}s}\ZZ_k^{(p)}(s)  \right|_{s=u_{1}} \right) \label{eqn:relative-entropy}
\end{align}
where $\ZZ_k^{(p)}(s)$ denotes the Cohen-Lenstra zeta function. 
Recall that $\ZZ_k^{(p)}(s)$ is defined as
$$\ZZ_k^{(p)}(s) := \sideset{}{'}\sum_{A} \frac{w_{k}(A)}{{\# A}^{s}}$$
with
\[ 
w_{k}(G) =
\begin{cases}
    w(G) \prod_{i = k-r+1}^{k} (1-p^{-i}) & \textrm{if } k \ge r := \mathrm{rank}(G),  \\
    0 & else
\end{cases}
\] 
and $w(G) := \frac{1}{\#\Aut(G)}$.
Note that $w_{k}(G)$ is increasing in $k$, $w_{k}(G)$ is bounded by $w(G)$, and $w_{k}(G)  \xrightarrow{k \to \infty} w(G) := \frac{1}{\# \Aut G}$. The Cohen-Lenstra zeta function converges for $\Re(s) > -1$ and satisfies the following explicit formula
\begin{equation} \label{eqn:CL-zeta}
\ZZ_k^{(p)}(s) = \prod_{i \ge 1}^{k} (1-p^{-s-i})^{-1}.
\end{equation}
We compute its derivative in two ways. Using the definition, we first find 
\[
\frac{\mathrm{d}}{\mathrm{d}s}  \ZZ_{k}^{(p)}(s) = - \sideset{}{'}\sum_{A} \frac{w_{k}(A) \log \#A}{{\# A}^{s}}
\]
since the series $\sideset{}{'}\sum_{A} \frac{w_{k}(A) \log \#A}{{\# A}^{s}}$ is absolutely uniformly convergent for $s \in [t,\infty)$ for any $t > -1$. This follows by comparing it to $$\sideset{}{'}\sum_{A} \frac{1}{{\# A}^{s-\varepsilon} \#\Aut A}$$ and using the proof of item I) of Theorem \ref{thm:main-entropy}. 
The limit interchange giving line (\ref{eqn:relative-entropy}) follows from Lebesgue's Dominated Convergence Theorem and the same comparison. 

On the other hand, by formula (\ref{eqn:CL-zeta}), we have: 
\[
\frac{\mathrm{d}}{\mathrm{d}s}  \ZZ_{k}^{(p)}(s) = \ZZ_{k}^{(p)}(s) \cdot \frac{\mathrm{d}}{\mathrm{d}s} \left( \log \ZZ_{k}^{(p)}(s) \right)  = \left( \prod_{i=1}^{k} (1-p^{-s-i})^{-1} \right) \left( - \sum_{i=1}^{k} \frac{\log(p)}{p^{s+i}-1} \right). 
\]
It follows that 
$$ \lim_{k \rightarrow \infty} \left. - \frac{\mathrm{d}}{\mathrm{d}s}\ZZ_k^{(p)}(s)  \right|_{s=u_{1}} = \frac{1}{F_{u_{1}}}\sum_{i=1}^{\infty}\frac{\log(p)}{p^{u_{1}+i}-1}, $$
which completes the proof.
\end{proof}

\section*{Acknowledgements}
The author thanks Akshay Venkatesh for his words of encouragement. 
The author was partially supported by an NSERC Postdoctoral Fellowship, the Institute for Advanced Study (through the National Science Foundation under Grant No. DMS-1926686), and Princeton University.

\nocite{*}
\bibliographystyle{plain}
\bibliography{bibfile.bib}

\end{document}